\documentclass[11pt,a4paper,reqno]{amsart}

\usepackage[utf8]{inputenc}

\usepackage{array,amsbsy,amscd,amsfonts,color,amssymb,amstext,amsmath,latexsym,amsthm,amscd
,multicol,graphicx,tikz}
\usepackage[english]{babel}
\usepackage{hyperref}

\newtheorem{theorem}{Theorem}
\numberwithin{theorem}{section}% [if desired]
\newtheorem{proposition}[theorem]{Proposition}
\newtheorem{lemma}[theorem]{Lemma}

\theoremstyle{definition}
\theoremstyle{definition}\newtheorem{definition}[theorem]{Definition}
\theoremstyle{definition}\newtheorem{remark}[theorem]{Remark}
\theoremstyle{definition}\newtheorem{example}[theorem]{Example}
\theoremstyle{definition}
\theoremstyle{definition}
\theoremstyle{definition}
\theoremstyle{definition}\newtheorem{problem}[theorem]{Problem}

\newcommand{\A}{\mathcal{A}}

\renewcommand{\H}{\mathcal{H}}

\renewcommand{\L}{\mathcal{L}}
\newcommand{\N}{\mathbb{N}}

\newcommand{\R}{\mathbb{R}}
\newcommand{\C}{\mathbb{C}}

\newcommand{\1}{\operatorname{id}}

\newcommand{\unit}{\mathbf{1}}

\newcommand{\bp}{\begin{proof}}
\newcommand{\ep}{\end{proof}}
\newcommand{\bdp}{\begin{dproof}}
\newcommand{\edp}{\end{dproof}}
\newcommand{\ra}{\rightarrow}

\newcommand{\ev}{\operatorname{ev}}

\newcommand{\Mat}{\operatorname{M}}

\newcommand{\tr}{\operatorname{tr }}

\newcommand{\rme}{\operatorname{e}}

\newcommand{\eps}{\varepsilon}

\newcommand{\diag}{\operatorname{diag}}

\renewcommand{\dim}{\operatorname{dim }}
\newcommand{\spa}{\operatorname{span }}

\newcommand{\id}{\operatorname{id}}

\newcommand{\cf}{{cf.\/}\ }

\title{Roots of Completely Positive Maps}
\author{B.V. Rajarama Bhat}
\address{Indian Statistical Institute, Stat-Math. Unit, R V College Post, Bengaluru 560059, India}
\author{Robin Hillier}
\address{Lancaster University, Department of Mathematics and Statistics, Lancaster LA1 4YF, United Kingdom}
\author{Nirupama Mallick}
\address{The Institute of Mathematical Sciences, IV Cross Road, CIT Campus, Taramani, Chennai 600113, India}
\author{Vijaya Kumar U.}
\address{Indian Statistical Institute, Stat-Math Unit, R V College Post, Bengaluru 560059, India}
\date{5 November 2019}


\begin{document}

\begin{abstract}
We introduce the concept of completely positive roots of completely positive maps on operator algebras. We do this in different forms: as asymptotic roots, proper discrete roots and as continuous one-parameter semigroups of roots. We present structural and general existence and non-existence results, some special examples in settings where we understand the situation better, and several challenging open problems. Our study is closely related to Elfving's embedding problem in classical probability and the divisibility problem of quantum channels.
\end{abstract}

\maketitle

\section{Introduction}

In many mathematical settings, the concept of a square-root or
higher order roots is familiar, e.g.~in the context of real
numbers, in the context of matrices, in the context of
real-valued functions, or measures. What all of these settings have in common is the underlying structure of a semigroup. To be slightly more formal, given a semigroup $(A,\star)$
and given $a\in A$ and $n\in \mathbb{N}$, we can ask whether there exists some $x\in A$ such that $a=
x\star x \star \cdots \star x$ ($x$ appearing $n$ times). Then we
may call $x$ an $n$-th root of $a$. If such a root exists for all
$n$, we would call $a$ infinitely divisible. We may also ask whether there is a one-parameter semigroup $(x_{t})_{t\in\R_+}$ in $A$ (namely $x_{s+t}=x_s\star x_t$, for all $s,t\in
\mathbb{R}_+$), such that $x_{t_0}=a$ for some $t_0>0.$ If this is the case then $a$ may be called embeddable (into a continuous semigroup). Finally, if there is a topology on $A$, we may also look for asymptotic roots or
asymptotic embeddability, that is, whether there is a one parameter
semigroup $(x_t)_{t\in\R_+}$, with $\lim _{t\ra\infty}x_t=a$.

Yuan \cite{Yu} deals with some of these questions in the full generality of topological semigroups, but without further structure it seems that one cannot say too much. In the present paper we would like to add such structure and look at unital normal completely positive (UNCP) maps on von Neumann algebras. They arise in many ways in operator algebras and quantum physics and are natural objects to
study (\cf \cite{EK,Sto,Wo}). Since  `composition' is an associative
operation on UNCP maps, it makes sense to study the question of roots
also in this setting, namely: given a von Neumann algebra $\A$, a number
$n\in\N$, and a UNCP map $\phi:\A\ra \A$, is there
another UNCP map $\psi:\A\ra \A$ such that $\phi =
\psi^n$? It turns out that currently surprisingly little is known in general.

However, there are a number of connection points with results in related areas. For example, we could specialize to commutative algebras. Then UNCP maps become stochastic maps of Markov chains in classical probability theory. A notable special case is that of (discrete or continuous) convolution semigroups of probability measures. Existence and uniqueness criteria for roots of
stochastic maps have been studied earlier (see e.g. \cite{HP,HL}).
Suppose a given Markov chain on a countable state space converges to an invariant distribution (an absorbing state). Then typically such a convergence happens exponentially (i.e., asymptotically) over time. In discrete time there are some instances when the convergence takes place in finite time.  An analysis of transition probability matrices of
such Markov chains can be seen in \cite{BG} and \cite{Su}. This does not work in continuous time if the semigroup generator is bounded, \cite{GI}.  But the condition of boundedness of the generator might be
too strong, and it is widely believed that under some very minimal continuity
assumptions on the transition semigroup, convergence in finite time should be impossible. However, we are not aware of any proof. Surprisingly the noncommutative counterpart is more involved and convergence in finite continuous time to a given pure state is indeed possible as has been shown in \cite{Bhat12}; more precisely, for a given normal pure state $\varphi$ on $B(\H)$, identified with the completely positive map
$\phi=\varphi(\cdot)\unit$, it is possible to construct a quantum
Markov semigroup $(\tau _t)_{t\in \mathbb{R}_+}$ (a strongly
continuous one-parameter semigroup of UNCP maps) on $B(\H)$ which coincides with $\phi$ at all times
$t\ge 1$; in other words, for all states $\psi $,  we get
$\psi\circ\tau_t = \varphi$, for $t\ge 1$. Hence convergence in the
continuous setting is possible in finite time, for all
pure states on $B(\H).$ A trivial consequence is that $\phi $  has
an $n$-th root, for every $n\in\N$. It is natural to ask what happens in the case  of
$\phi=\varphi(\cdot)\unit$ where $\varphi$ is a mixed state. What
can be said about $n$-th roots or semigroups of roots of $\phi$? And
in light of our first observation, what can be said about $n$-th
roots and semigroups of roots of more general completely positive
maps $\phi$, not only those arising from states?

If we drop the assumption of convergence to an invariant distribution, many things can happen. E.g.~the question of continuous roots of a given stochastic map makes sense here and is known under the name of Elfving's embedding problem, dating back to 1937 \cite{Elf}: given a stochastic map $S$, when can we find a map $L$ such that $\rme^{L}=S$ and all $\rme^{tL}$, for $t\ge 0$, are stochastic maps? A number of necessary and sufficient criteria have been found over the years, see e.g.~\cite{Dani,Dav2,Kin,vB} for a non-exhaustive list. One particularly interesting condition is: if $S$ is infinitely divisible, i.e. it has $n$-th roots for all $n$, then it is embeddable into a semigroup \cite{Kin}. The noncommutative analogue of this question is not so easy but we may restrict ourselves to finite dimensions to start with. Indeed it should be noted that completely positive (trace-preserving) maps in finite dimensions form the basis of quantum information theory (there termed ``quantum channels" \cite{Wo}). The question of asymptotic behavior of sequences of compositions of quantum channels appears relevant in quantum information problems, e.g.~in the context of entanglement breaking maps \cite{RJP}, and the question of ``divisibility" of quantum channels has also been studied in a  few places. A number of divisibility criteria can be found in \cite{Den,WC,WECC,BC} and some of the questions we pose here have also been discussed in \cite{WC,BC} but with slightly different terminology and complementary answers. Most notably though, it has been shown that the complexity of the problems of deciding whether a given quantum channel (or stochastic map) has a square-root or whether it is embeddable into a continuous semigroup are NP-hard \cite{CEW1,CEW2,BC}. This means that the set of such quantum channels has no simple expression other than explicit enumeration of its elements, and it is impossible to find a simply verifiable criterion for the existence of such roots as in the case of positive numbers or matrices, for example. However, this should not discourage us from looking for interesting new relations or characterizations, at least for some special classes of UNCP maps, and that is what we would like to do here.

Our outline is as follows. In Section \ref{sec:asymptotic}, we start by describing the quantum analogue and generalization of the ``exponential" convergence to a given invariant distribution: given a UNCP map, is there a continuous one-parameter semigroup that converges to $\phi$ as $t\ra\infty$? We completely clarify this question. Such semigroups we will call asymptotic continuous roots. As a byproduct we obtain an affirmative answer to a question of Arveson (Problem 3 in \cite[p.387]{Arv-book}) through very elementary methods.

We then move on to the question of proper roots in the finite-time setting, where Section \ref{sec:discrete} deals with the $n$-th root case while Section \ref{sec:continuous} deals with the continuous semigroup case. We are able to provide several existence and non-existence results under different additional assumptions, e.g.~regarding the dimension or structure of the algebra or the range of the CP map. In particular, for the case of states on $\Mat_d$ or $B(\H)$ or $\C^d$ we have a complete characterization of existence of $n$-th roots. However, we are still far from a full understanding and have to leave some questions in the form of conjectures and difficult open problems. We hope some readers will feel stimulated to think about them and come up with nice solutions. Our diversity of results together with the few related results in literature, e.g.~\cite{HL,CEW1,CEW2,BC}, indicate that a ``complete and elegant" characterization is unlikely to be found though.

In this paper we are mainly concerned with the existence or non-existence of roots of UNCP maps. Whenever such roots exist, they are typically far from unique, and a subsequent natural question would be to find a useful characterization of all such roots for a given UNCP map. We deal with UNCP maps (in finite and infinite dimension) and to some minor extent with the commutative special case of stochastic maps though we do not look at the related question of nonnegative roots of (entry-wise) nonnegative matrices as can be found in other places, e.g.~\cite{Mi}.

Throughout this paper, all our C*-algebras or von Neumann algebras are supposed to live on some Hilbert space, and all Hilbert spaces here are taken as complex and separable, sometimes even finite-dimensional, with scalar products linear in the second variable. $B(\H)$ will denote the bounded linear operators on the Hilbert space $\H$ and $\Mat_d$ the $d\times d$ matrices with complex entries; $\N$ denotes the natural numbers without 0. For an introduction to completely positive maps we refer the reader to any good textbook on the topic, e.g.~\cite{EK,Sto,Tak}.

\section{Asymptotic roots}\label{sec:asymptotic}

In the present section we work in the C*-algebraic setting because it appeared more natural to us; however, everything can be adjusted and translated in a straight-forward way to the von Neumann algebraic setting, cf.~also Remark \ref{rem:asymptotic} below, which would also bring it more in line with the subsequent sections.

\begin{definition}\label{def:asymptotic}
Given a unital C*-algebra $A$ and a bounded unital completely positive (UCP) map $\phi:A\ra A$,
\begin{itemize}
\item[(ad)] an \emph{asymptotic discrete root} of $\phi$ is a UCP map $\tau:A\ra A$ such that $\tau^n\ra \phi$ (pointwise in norm), as $n\ra \infty$;
\item[(ac)] an \emph{asymptotic continuous root} of $\phi$ is a uniformly continuous one-parameter semigroup $(\tau_t)_{t\ge 0}$ of UCP maps on $A$ such that $\tau_t\ra \phi$ (pointwise in norm), as $t\ra \infty$.
\end{itemize}
\end{definition}

We then have:

\begin{theorem}\label{th:asymptotic}
Let $A$ be a unital C*-algebra and $\phi$ a UCP map of $A$. Then the following three statements are equivalent:
\begin{itemize}
\item[(i)] $\phi$ is idempotent, i.e., $\phi^2=\phi$;
\item[(ii)] $\phi$ has an asymptotic continuous root;
\item[(iii)] $\phi$ has an asymptotic discrete root.
\end{itemize}
\end{theorem}

\begin{proof}
(i) $\Rightarrow$ (ii). Suppose $\phi^2=\phi$. Then define the map
\[
\L:= \phi-\id = -(\id-\phi): A\ra A,
\]
which is bounded and conditionally completely positive \cite[Sec.4.5]{EK} and therefore generates a uniformly continuous UCP semigroup $\tau$. We find
\[
\L^n(x) = (-1)^n (\id-\phi)^n(x)= (-1)^n (\id-\phi)(x), \quad x\in A,
\]
and therefore
\[
\tau_t(x) = \sum_{n=0}^\infty \frac{(-t)^n (\id-\phi)^n}{n!} x
= \sum_{n=0}^\infty \frac{(-t)^n}{n!} (\id-\phi) (x) + \phi(x)
= \rme^{-t} (\id-\phi)(x) + \phi(x).
\]
We see
\begin{equation}\label{eq:asymptotic}
\|\tau_t(x)-\phi(x)\| = \rme^{-t} \| x-\phi(x)\| \le \rme^{-t} (1+\|\phi\|) \|x\|, \quad x\in A,
\end{equation}
so $\tau_t\ra \phi$ uniformly, as $t\ra\infty$, so $\tau$ is an asymptotic continuous root for $\phi$.

(ii) $\Rightarrow$ (iii) is obvious.

(iii) $\Rightarrow$ (i). Suppose now that there is an asymptotic discrete root $\tau$ of $\phi$. Then the fact that $\tau^m \ra \phi$ as $m\ra\infty$ allows us to make the following manipulations:
\[
\phi\circ\tau^n(x)= \lim_{m\ra\infty} \tau^m \circ \tau^n(x)= \lim_{n+m\ra \infty} \tau^{n+m} (x)= \phi(x), \quad x\in A,
\]
and therefore
\[
\phi^2 (x)= \lim_{m\ra\infty} \phi\circ \tau^m (x)= \lim_{m\ra\infty} \phi (x)= \phi(x), \quad x\in A,
\]
so $\phi^2=\phi$.
\end{proof}

\begin{remark}\label{rem:asymptotic}
Let $A$ be a unital C*-algebra and $\phi$ a UCP map of $A$.
\begin{itemize}
\item[(i)] An asymptotic root of $\phi$ is in general not unique.
\item[(ii)] We did not specify the dimension of $A$ and the Hilbert space $\H$ on which it acts. In fact, the statements are interesting in both finite and infinite dimensions.
\item[(iii)] If $\phi$ has an asymptotic (discrete/continuous) root with respect to the strong operator topology then the above proof shows that $\phi$ also has an asymptotic (discrete/continuous) root with respect to the uniform topology.
\item[(iv)]  The definition, theorem and proof continue to hold true upon replacing C*-algebras by von Neumann algebras, UCP by UNCP, and the uniform by the strong operator topology.
\end{itemize}
\end{remark}

\begin{remark}
As a byproduct, the theorem answers Problem 3 in \cite[p.387]{Arv-book} affirmatively, namely given an eigenvalue list $(\lambda_1,\lambda_2,\ldots)$ with $0\le\lambda_i\le 1$ and $\sum_i \lambda_i=1$ as in \cite[Sec.12.4]{Arv-book}, consider the density matrix
\[
D:= \diag(\lambda_1,\lambda_2,\ldots)\in B(\H),
\]
the normal state $\varphi=\tr(D\cdot)$ on $B(\H)$, and the UNCP map
\[
\phi = \varphi(\cdot) \unit: B(\H) \ra B(\H).
\]
Then the asymptotic root in the above proof is a UNCP semigroup with bounded generator that has $\varphi$ as absorbing state: for every normal state $\psi$ on $B(\H)$ and every $x\in B(\H)$, we get
\[
|\psi\circ\tau_t(x) - \varphi(x) | = | \psi(\tau_t(x)-\varphi(x)\unit)|
\le \| \tau_t(x)-\varphi(x)\unit \| \le 2 \rme^{-t} \|x\|
\]
where the last inequality follows from \eqref{eq:asymptotic}; thus,
\[
\|\psi\circ\tau_t - \varphi \| \ra 0, \quad t\ra\infty,
\]
meaning that $\varphi$ is an absorbing state for $(\tau_t)_{t\ge 0}$, which answers Problem 3 in \cite[p.387]{Arv-book}.
\end{remark}

\section{Proper discrete roots}\label{sec:discrete}

In this and the following section, we work exclusively with von Neumann algebras.

\subsection{General statements}

Our fundamental definition is the following:

\begin{definition}\label{def:discrete}
Given a von Neumann algebra $\A$, a UNCP map $\phi:\A\ra \A$ and an integer $n\in\N\setminus\{1\}$, a \emph{proper $n$-th discrete root} of $\phi$ is a UNCP map $\tau:\A\ra \A$ such that $\tau^n= \phi$ and $\tau^k\not=\phi$ for all $k<n$. We call $n$ the \emph{order} of $\tau$.
\end{definition}

We need a notational convenience which turns out to be crucial in many proofs and characterizations:

\begin{definition}\label{def:support-rank}
For every UNCP map $\phi$ on a von Neumann algebra $\A$, we define the \emph{support projection} as the smallest projection $p_\phi\in\A$ such that $\phi(p_\phi)=\unit$. We write $p_\phi':=\unit -p_\phi\in\A$.
\end{definition}

The existence and the uniqueness of $p_\phi$ follow from \cite[Prop.I.4.3, p.63]{Dix}, roughly as follows: one first realizes that the set of $x\in\A$ such that $\phi(x^*x)= 0$ forms a $\sigma$-weakly closed left ideal in $\A$. For such ideals there exists a maximal projection $p$ such that the ideal consists of all $x\in\A$ with $x=xp$. This is exactly the projection $p=p_\phi'=\unit-p_\phi$ from the preceding definition.

We use the following block matrix decomposition of $x\in \A$:
\begin{equation}\label{eq:block-dec}
x = \begin{pmatrix}
 x_{11} & x_{12} \\ x_{21} & x_{22}
 \end{pmatrix}
=\begin{pmatrix}
  p_\phi x p_\phi & p_\phi x p_\phi' \\ p_\phi' x p_\phi & p_\phi' x p_\phi'
 \end{pmatrix}
.
\end{equation}

A first useful fact is the following variation of \cite[Th.4.2]{BM} about the relation with nilpotent NCP maps:

\begin{lemma}\label{lem:nilpotent}
Let $\A$ be a von Neumann algebra, $\phi$ a UNCP map of $\A$ and $n\in\N$. Suppose there exists a proper $n$-th discrete root $\tau$ of $\phi$. Then
\begin{itemize}
\item[(i)] $\tau(p_\phi) \ge p_\phi$;
\item[(ii)]  there also exists a nilpotent NCP map $\alpha: p_\phi' \A p_\phi' \ra p_\phi' \A p_\phi'$ of order at most $n$ such that
\[
 \tau\begin{pmatrix}
  0 & 0 \\ 0 & x
 \end{pmatrix}
 =\begin{pmatrix}
  0 & 0 \\ 0 & \alpha(x)
 \end{pmatrix},
 \quad x\in p_\phi' \A p_\phi';
\]
\item[(iii)] for every $\begin{pmatrix}
0 & x \\ y & z
\end{pmatrix}\in \A
$ w.r.t.~to the above block decomposition, there is $\begin{pmatrix}
0 & x' \\ y' & z'
\end{pmatrix} \in \A
$
such that
\[
\tau\begin{pmatrix}
  0 & x \\ y & z
 \end{pmatrix}
 =\begin{pmatrix}
  0 & x' \\ y' & z'
 \end{pmatrix};
\]
\item[(iv)] $p_\phi \tau(\cdot)p_\phi$ restricts to a proper discrete root of $p_\phi \phi(\cdot )p_\phi $ on $p_\phi \A p_\phi$ of order at most $n$.
\end{itemize}
\end{lemma}

\begin{proof}

(i). We first notice that $\tau(p_\phi)\le \unit$ since $\tau$ was assumed to be UNCP. Therefore $0\le p_\phi \tau(p_\phi) p_\phi \le p_\phi$. Let us write $b:= p_\phi-p_\phi \tau(p_\phi) p_\phi \ge 0$. We would like to show that $b=0$.

To start with,
\[
\tau\circ\phi= \tau\circ \tau^n = \tau^{n+1} = \tau^n\circ\tau=\phi\circ\tau.
\]
This implies
\[
\phi(p_\phi)=\unit= \tau(\unit) = \tau(\phi (p_\phi)) = \phi(\tau (p_\phi)) = \phi(p_\phi\tau(p_\phi)p_\phi),
\]
thus
\[
\phi(b)= \phi(p_\phi - p_\phi\tau(p_\phi)p_\phi) = 0.
\]

Let $e_b\in \A$ be the support projection of $b$, which can be defined through Borel functional calculus. Notice that $e_b\le p_\phi$ because $b\le p_\phi$, so $p_\phi-e_b$ is a subprojection of $p_\phi$. Then it follows from the construction of the spectral theorem (in projection-valued measures form \cite[Sec.7.3]{RS}) that $\phi(b)=0$ if and only if $\phi(e_b)= 0$. Since we have already proved $\phi(b) = 0$, we find
\[
\phi(p_\phi-e_b) = \phi(p_\phi) - \phi(e_b) = \unit - 0 = \unit.
\]
Thus $p_\phi-e_b$ fulfills the properties of a support projection of $\phi$ and therefore must be equal to $p_\phi$ due to its uniqueness, so $e_b=0$, hence $b=0$.

(ii). Unitality of $\tau$ together with part (i) implies $\tau(p_\phi')\le p_\phi'$. Thus, $\tau(p_\phi' \cdot p_\phi')$ is a NCP map with image in $p_\phi'\A p_\phi'$, hence giving rise to an NCP map
\[
\alpha:= \tau\restriction _{p_\phi'\A p_\phi'}:  p_\phi'\A p_\phi'\ra p_\phi'\A p_\phi'.
\]
Since $\tau$ is an $n$-th root of $\phi$, we have
\[
 0 = \phi\begin{pmatrix}
  0 & 0 \\ 0 & x
 \end{pmatrix}
 = \tau^n \begin{pmatrix}
  0 & 0 \\ 0 & x
 \end{pmatrix}
 =\begin{pmatrix}
  0 & 0 \\ 0 & \alpha^n(x)
 \end{pmatrix},
 \quad x\in p_\phi'\A p_\phi',
\]
implying that $\alpha$ is nilpotent of order at most $n$.

(iii) Part (i) shows that $\tau(p_\phi') \le p_\phi'$. Using the block decomposition in \eqref{eq:block-dec}, we can write
\[
 \begin{pmatrix}
  0 & 0 \\ 0 & \unit
 \end{pmatrix}
 \ge \tau\begin{pmatrix}
  0 & 0 \\ 0 & \unit
 \end{pmatrix}
 \ge \tau\Big[ \begin{pmatrix}
  0 & 0 \\ x^* & 0
 \end{pmatrix}
 \begin{pmatrix}
  0 & x \\ 0 & 0
 \end{pmatrix} \Big]
 \ge \tau\begin{pmatrix}
  0 & x \\ 0 & 0
 \end{pmatrix}^*
 \tau\begin{pmatrix}
  0 & x \\ 0 & 0
 \end{pmatrix},
\]
for every $x\in p_\phi \A p_\phi'$ with $x^*x\le p_\phi'$. This means that
\[
 \tau\begin{pmatrix}
  0 & x \\ 0 & 0
 \end{pmatrix}
 =\begin{pmatrix}
  0 & x' \\ 0 & z'
 \end{pmatrix}
\]
with certain $x'\in p_\phi \A p_\phi', z'\in p_\phi' \A p_\phi'$. Together with part (ii) and the self-adjointness of $\tau$, we have, for any $x\in p_\phi \A p_\phi',y\in p_\phi' \A p_\phi, z\in p_\phi' \A p_\phi'$:
\[
 \tau\begin{pmatrix}
  0 & x \\ y & z
 \end{pmatrix}
 =\begin{pmatrix}
  0 & x' \\ y' & z'
 \end{pmatrix}
\]
with certain $x'\in p_\phi \A p_\phi',y'\in p_\phi' \A p_\phi, z'\in p_\phi' \A p_\phi'$.

(iv) Since $p_\phi\in\A$ and $p_\phi\phi(p_\phi)p_\phi = p_\phi= p_\phi\tau(p_\phi)p_\phi$ by part (i), it is clear that both $p_\phi\phi(\cdot)p_\phi$ and $p_\phi\tau(\cdot)p_\phi$ restrict to UNCP maps on $p_\phi\A p_\phi$. Moreover, it follows from part (iii) that
\[
p_\phi\tau(p_\phi\tau(\cdot)p_\phi)p_\phi = p_\phi\tau^2(\cdot)p_\phi,
\]
and by induction, since $\tau^n=\phi$, we get that $p_\phi\tau(\cdot)p_\phi$ is a proper discrete root of $p_\phi\phi(\cdot)p_\phi$ of order at most $n$.
\end{proof}

If $\phi$ is idempotent then there is generally more hope to say something about roots. A particularly nice case of idempotency is that where $\phi$ has rank one, namely $\phi=\varphi(\cdot)\unit$ for some normal state $\varphi$ on $\A$. In that case, we get the following easy correspondence:

\begin{lemma}\label{lem:discrete-rootnilpotent}
Given a von Neumann algebra $\A$, a normal state $\varphi$ on $\A$
and $n\in\N$, let $\phi=\varphi(\cdot)\unit$, which is UNCP. Then a
 map $\tau$ on $\A$ is a proper $n$-th discrete root of $\phi$ if
and only if $\tau=\phi+\alpha$ with $\alpha$ some normal nilpotent
map of order $n$ such that $\alpha\circ\phi=0=\phi\circ\alpha$ and
$\phi+\alpha$ is CP.
\end{lemma}
\begin{proof}
    $(\Rightarrow)$
    Consider $\alpha:=\tau-\phi.$ Clearly $\alpha$ is normal. Since $\phi\circ \tau=\tau\circ \phi=\tau^{n+1}=\phi$ we have, for all $k\ge 1$
    $\alpha^k=\tau^k-\phi.$ Now since $\tau^k\ne \phi$ for $k<n$ and $\tau^n=\phi$ we have $\alpha^k\ne 0$ for $k<n$ and $\alpha^n=0.$ i.e., $\alpha$ is nilpotent  of order $n. $  Also $\phi\circ \alpha=\phi\circ(\tau-\phi)=0=(\tau-\phi)\circ \phi=\alpha\circ\phi.$ The converse part $(\Leftarrow)$ is trivial.
\end{proof}

When $p:=\dim \A<\infty $, Lemma \ref{lem:discrete-rootnilpotent} shows that any UNCP map arising from a state on $\A$ cannot have proper discrete roots of order higher than $p$. Indeed the following lemma shows that the order of such a root must be strictly less than $p$:

\begin{lemma}\label{lem:discrete-rootnilpotent-finite-dim}
Let $\A$ be a finite dimensional von Neumann algebra of dimension $p$. Let  $\tau$ be a UNCP map on $\A$. Then the following are equivalent:
\begin{itemize}
\item[(i)] $ \tau^ n=\phi:=\varphi(\cdot)\unit $ for some state $\varphi$ on $\A$ and for some $n\in \N $.
\item[(ii)] $\tau =\phi +\alpha$ for some nilpotent map $\alpha$ and $\phi:=\varphi(\cdot)\unit$ for some state $\varphi$ with $\alpha \circ \phi=0=\phi\circ \alpha $.
\item[(iii)] $0$ is an eigenvalue of $\tau$ with algebraic multiplicity $p-1$.
\item[(iv)] $\tr  \tau^k=1$ for all $k\ge 1$.
\end{itemize}
In any of these equivalent cases, $\tau$ is a root of order at most
$p-1$.
\end{lemma}

\begin{proof}
The idea of the proof is to treat $\phi$ and $\tau$ as linear maps on $\C^p$.

(i) $\Leftrightarrow $ (ii) follows from Lemma \ref{lem:discrete-rootnilpotent}.

(i) $\Rightarrow$ (iii). As $\tau ^n=\phi$ has rank $1,$ $0$ is an eigenvalue of $\tau^n$ of multiplicity $p-1,$ hence  $0$ is an eigenvalue of $\tau $ of multiplicity $p-1.$

(iii) $\Rightarrow$ (i). Looking at the Jordan normal form of $\tau$ it is clear that  $\tau ^n$ has rank $1$ for some $n\in \N.$ Since $\tau^n$ is unital, there is a state $\varphi$ on $\A$ such that $\tau^n(x)=\varphi(x)\unit$ for all $x\in \A$.

(iii) $\Rightarrow$ (iv) is obvious as $\tau (\unit)=\unit.$

(iv) $\Rightarrow$ (iii).  Let $\lambda_1=1,\lambda_2,
\lambda_3,...,\lambda_m$ be the distinct  eigenvalues of $\tau$ with
algebraic multiplicity $a_1,a_2,...,a_m$, respectively.  From (iv)
we have $(a_1-1)\lambda_1^k+a_2\lambda_2^k+\cdots +a_m\lambda_m^k=0$
for all $k\ge 1.$  Consider the Vandermonde matrix
    $V:=(\lambda_i^{j-1})_{1\leq i,j\leq n}\in \Mat_m.$ Then as the $\lambda_i$'s are mutually distinct, we have $\det V\ne 0.$ Also note that  $V((a_1-1)\lambda_1, a_2\lambda_2,...,a_m\lambda_m)'=0. $ This implies that  $((a_1-1)\lambda_1, a_2\lambda_2,...,a_m\lambda_m)=0$. Hence $a_1=1,m=2$ and $\lambda_2=0.$ That means $0$ is an eigenvalue of $\tau$ with algebraic multiplicity $p-1.$

Now regarding our final statement, let $\tau$ be a proper $n$-th discrete root of $\phi=\varphi(\cdot)\unit$ on $\A.$ It is clear from (iii) $\Leftrightarrow$ (i) that  $n$ is the maximal possible size of all Jordan blocks of $\tau$. Hence $n\le p-1$.
\end{proof}

\begin{remark}
It is worth pointing out that a proper $n$-th discrete root $\tau$ for a state $\varphi$ is ``absorbing", namely $\psi\circ\tau^k=\varphi$, for all $k\ge n$ and all other states $\psi$. So in this case $\tau$ is also an asymptotic discrete root. The same is true for proper versus asymptotic continuous roots, as shall become clear from the following section, cf.~Proposition \ref{prop:cont-state}. In general though, there is no clear relationship between proper and asymptotic roots.
\end{remark}

Here are some examples regarding existence and non-existence of roots of UNCP maps in finite dimensions.
We start with a map which has no nontrivial proper discrete roots at all.

\begin{example}\label{ex:discrete-UCP1}
 Let $\phi:\Mat_2\ra\Mat_2$ be the UNCP map defined by
 $\phi
\begin{pmatrix}
 a&b\\c&d
\end{pmatrix}
=\begin{pmatrix}
d&0\\0&a
\end{pmatrix}$. We claim that $\phi$ has no proper discrete root. Suppose for contradiction there exists a proper $n$-th discrete root $\tau$ for $\phi$, then $\tau^n=\phi$ and $\tau\circ\phi=\tau\circ\tau^n=\tau^{n+1}=\tau^n\circ\tau=\phi\circ\tau$.
Let
\begin{gather*}
\tau\begin{pmatrix}
1&0\\0&0\end{pmatrix}=
\begin{pmatrix}a_{11}&a_{12}\\
a_{21}&a_{22}
\end{pmatrix}, \quad
\tau\begin{pmatrix}
0&1\\0&0\end{pmatrix}=
\begin{pmatrix}b_{11}&b_{12}\\
b_{21}&b_{22}
\end{pmatrix},\\
\tau\begin{pmatrix}
0&0\\1&0\end{pmatrix}=
\begin{pmatrix}c_{11}&c_{12}\\
c_{21}&c_{22}
\end{pmatrix},\quad
\tau\begin{pmatrix}
0&0\\0&1\end{pmatrix}=
\begin{pmatrix}d_{11}&d_{12}\\
d_{21}&d_{22}
\end{pmatrix}.
\end{gather*}
Since $\tau\circ\phi=\phi\circ\tau$ and $\tau(\unit)=\unit$, we have $a_{12}=a_{21}=d_{12}=d_{21}=0$ and $a_{11}=d_{22}$ and $d_{11}=a_{22}\not=0$. It follows that
\begin{gather*}
\begin{pmatrix}                                                                                       0&0\\0&1                                                                                        \end{pmatrix}
=\tau^n\begin{pmatrix}
1&0\\0&0\end{pmatrix}=
\begin{pmatrix}a_{11}^n+*&0\\
0&*
\end{pmatrix}\\
\begin{pmatrix}                                                                                       1&0\\0&0                                                                                        \end{pmatrix}
=\tau^n\begin{pmatrix}
0&0\\0&1\end{pmatrix}=
\begin{pmatrix}*&0\\
0&d_{22}^n+*
\end{pmatrix},
\end{gather*}
where all $*$'s are nonnegative terms depending on $a_{11}$ and $a_{22}$ only. In particular we see from these equalities that $a_{11}=d_{22}=0$ and the only possible solution is
\[
\tau\begin{pmatrix}
1&0\\0&0
\end{pmatrix}=
\begin{pmatrix}
0&0\\0&1
\end{pmatrix}, \quad
\tau\begin{pmatrix}
0&0\\0&1\end{pmatrix}=
\begin{pmatrix}
1&0\\0&0
\end{pmatrix},
\]
i.e., $\tau=\phi$. Thus $\phi$ has no proper $n$-th discrete root.
\end{example}

The following map has only a proper square root.

\begin{example}\label{ex:discrete-UCP2}
 Let $\phi:\Mat_2\ra\Mat_2$ be the idempotent UNCP map defined by
 $\phi
\begin{pmatrix}
 a&b\\c&d
\end{pmatrix}
=\begin{pmatrix}
a&0\\0&d
\end{pmatrix}$. Then $\phi$ has a proper square root $\tau
\begin{pmatrix}
 a&b\\c&d
\end{pmatrix}
=\begin{pmatrix}
d&0\\0&a
\end{pmatrix}$ but $\phi$ has no other proper discrete roots, which can be proved in the same style as Example \ref{ex:discrete-UCP1}.
\end{example}

Finally, a map with proper discrete roots of all orders:

\begin{example}\label{ex:discrete-UCP3}
 Let $\phi:\Mat_2\ra\Mat_2$ be the UNCP map defined by
 \[
 \phi
\begin{pmatrix}
 a&b\\c&d
\end{pmatrix}
=\begin{pmatrix}
a&\frac{b}{2}\\\frac{c}{2}&d
\end{pmatrix}.
\]
For every $n\in\N$, define
\[
\tau_{1/n}:\Mat_2\ra\Mat_2,\quad
\tau_{1/n}\begin{pmatrix}
 a&b\\c&d
\end{pmatrix}
=\begin{pmatrix}
a&\frac{b}{2^{\frac{1}{n}}}\\\frac{c}{2^{\frac{1}{n}}}&d
\end{pmatrix}.
\]
Then $\tau_{1/n}$ is a UNCP map and $\tau_{1/n}^n=\phi$, so $\tau_{1/n}$ is a proper $n$-th discrete root of $\phi$.
\end{example}

\begin{example}\label{ex:discrete-UCP4}
Let $\phi:\Mat_2\ra\Mat_2$ be the UNCP map defined by
\[
\phi
\begin{pmatrix}
a&b\\c&d
\end{pmatrix}
=\begin{pmatrix}
d&\frac{c}{2}\\\frac{b}{2}&a
\end{pmatrix}.
\]
Then $\phi$ has an $n$-th root for every odd $n\in\N\setminus\{1\}$ but not for even $n$. This is again proved in the same way as Example \ref{ex:discrete-UCP1}.
\end{example}

So we are led to the following problem:

\begin{problem}\label{prob:discrete}
Suppose $\A=\Mat_d$ or $B(\H)$ and  $\phi$ is a UNCP map on $\A$. Then for which $n\in\N$ is there a proper $n$-th discrete root of $\phi$?
\end{problem}

Though we have got some illustrative examples here, a general characterization of existence and non-existence of proper discrete roots is expected to be complicated and does involve more details about the map $\phi$, as the following subsection indicates. Similar facts have been pointed out in \cite{BC} and it matches the findings in \cite[Sec.4]{HL}.

\subsection{Proper discrete roots for states on $\Mat_d$ and $B(\H)$}

We can say much more by specializing the results of the preceding subsection to the setting of normal states on $B(\H)$ or $\Mat_d$, which we are going to do now.

\begin{theorem}\label{th:discrete-fin-state}
Suppose $d<\infty$ and $\varphi$ is a state on $\Mat_d$ of support dimension $r:=\dim(p_\phi\C^d)$. Then $\phi=\varphi(\cdot)\unit$ has a proper $n$-th discrete root on $\Mat_d$ if and only if $1< n\le d+ r^2-r -1$.
\end{theorem}

\begin{proof}
We split the proof into two steps, depending on $r$. First of all, we may choose and fix a basis $(e_k)_{k=1,\ldots,d}$ such that $\varphi$ is in diagonal form, so $\varphi = \sum_{k=1}^d \lambda_k \langle e_k, \cdot e_k\rangle$ and $\lambda_1 \ge \ldots \ge \lambda_r > \lambda_{r+1} =0 = \ldots = \lambda_d$.

(Step 1) Suppose $r=d$, so $\varphi$ is faithful. We have to prove that $\phi$ has a proper $n$-th discrete root if and only if $1< n\le d^2-1$. First we see from Lemma \ref{lem:discrete-rootnilpotent-finite-dim} that if $\tau$ is a proper $n$-th discrete root of $\phi$ then $n\le d^2-1$. We write
\[
\alpha:=\tau-\phi,
\]
which is nilpotent of order $n$ with $\alpha(\unit)=0=\phi\circ\alpha$ owing to Lemma \ref{lem:discrete-rootnilpotent}.

Let us introduce the scalar product
\[
\langle\cdot,\cdot\rangle_{\varphi} : (x,y)\in \Mat_d\times \Mat_d \mapsto \varphi(x^* y).
\]
Then $\alpha$ restricts to a linear nilpotent map from $\Mat_d\ominus \C \unit$ into itself, and this subspace has dimension $d^2 -1$. An upper bound on the order of nilpotency is therefore $d^2-1$, so $n\le d^2-1$.

Next we would like show that we can actually attain this upper bound. To this end, consider
an orthonormal basis $(\unit, Y_1,\ldots Y_{d^2-1})$ of $\Mat_d$ with respect to $\langle\cdot,\cdot\rangle_{\varphi}$ such that $Y_i^*=Y_i$ and $\phi(Y_i)=0$, for all $i$, which can always be achieved. Then define
\[
\alpha(\unit) = \alpha(Y_{d^2-1})= 0, \quad \alpha(Y_{i}) = \eps Y_{i+1}, \quad i=1,\ldots , d^2-2,
\]
with suitable $\eps>0$ still to be determined, and
\[
\tau:=\phi + \alpha.
\]
Then it is clear that $\alpha$ is nilpotent of order $d^2-1$ and so $\tau^{d^2-1}=\phi$ because $\phi\circ\alpha=\alpha\circ\phi$ but $\tau^k\not=\phi$ for $k<d^2-1$. Moreover,  $\alpha$ is self-adjoint, namely $\alpha(x^*)=\alpha(x)^*$ for all $x\in\Mat_d$, thus is $\tau$. In order to show that $\tau$ is a proper discrete root, it remains to show that $\tau$ is CP. To this end, we compute the Choi matrix $C_\tau\in\Mat_d(\Mat_d)$ of $\tau$, cf.~\cite{Sto}, and find
\[
C_\tau = \begin{pmatrix}
\tau(e_{11}) & \ldots & \tau(e_{1 d}) \\
\vdots & & \vdots \\
\tau(e_{d1}) & \ldots & \tau(e_{d d})
\end{pmatrix}
=
\begin{pmatrix}
\tau(e_{11}^*) & \ldots & \tau(e_{ d1}^*) \\
\vdots & & \vdots \\
\tau(e_{1d}^*) & \ldots & \tau(e_{d d}^*)
\end{pmatrix}
=
\begin{pmatrix}
\tau(e_{11})^* & \ldots & \tau(e_{d 1})^* \\
\vdots & & \vdots \\
\tau(e_{1d})^* & \ldots & \tau(e_{d d})^*
\end{pmatrix}
= C_\tau^*
\]
so $C_\tau$ is self-adjoint for all $\eps$. We notice that $C_\tau$ depends continuously on $\eps$ and that for $\eps=0$, we get
\[
\begin{pmatrix}
\lambda_1 \unit & 0& \ldots & 0   \\
\vdots & && \vdots \\
0 & \ldots & 0& \lambda_d \unit
\end{pmatrix}.
\]
This matrix lies in the interior of the convex cone of positive matrices because all $\lambda_i>0$. Choosing $\eps>0$ small enough, we therefore find that $C_\tau$ must still be inside this cone. By Choi's theorem, cf.~\cite{Sto}, this implies that $\tau$ is CP, hence it is a proper discrete root of order $d^2-1$.

In order to get a proper discrete root of order $n<d^2-1$, all we have to do is change the map $\alpha$ accordingly, e.g
\[
\alpha(\unit) = \alpha(Y_n)=  \ldots= \alpha(Y_{d^2-1})=0, \quad \alpha(Y_{i}) = \eps Y_{i+1}, \quad i=1,\ldots , n-1,
\]
and proceed in the same way as above.

(Step 2) Next we examine the case $r<d$ and write $\Mat_r$ for $p_\phi\Mat_d p_\phi$. Suppose $\tau$ is a root of $\phi$. Then by Lemma \ref{lem:nilpotent}(iv),
\[
\tau' := p_\phi \tau(\cdot) p_\phi: \Mat_r \ra \Mat_r
\]
defines a proper discrete root of the faithful state $\varphi\restriction_{\Mat_r}$ on $\Mat_r$, hence its maximal order is $r^2-1$ according to (Step 1) above. As shown in Lemma \ref{lem:nilpotent}(iii), we have the following action in block decomposition:
\[
\tau \begin{pmatrix}
w & x \\ y& z
\end{pmatrix}
= \begin{pmatrix}
\tau'(w) & * \\ * & *
\end{pmatrix},
\]
in particular
\[
\tau^{r^2-1} \begin{pmatrix}
w & x \\ y& z
\end{pmatrix}
= \begin{pmatrix}
\varphi\restriction_{\Mat_r}(w) \unit & * \\ * & *
\end{pmatrix}.
\]
We therefore have to find the minimal number $n'$ such that
\begin{equation}\label{eq:discrete-n'}
\tau^{n'} \begin{pmatrix}
\unit & x \\ y& z
\end{pmatrix}
= \begin{pmatrix}
\unit & 0 \\ 0 & \unit
\end{pmatrix},
\end{equation}
for all $x,y,z$, and we claim that it can be at most $d-r$.

To this end, let us write
\[
\tau= \sum_{i=1}^N L_i^* (\cdot) L_i, \quad L_i=\begin{pmatrix}
A_i & B_i \\ C_i& D_i
\end{pmatrix}.
\]
Since
\[
0=\varphi\circ\tau\begin{pmatrix}
0 &  0 \\ 0 & \unit
\end{pmatrix}
= \sum_{i=1}^N \varphi \begin{pmatrix}
C_i^*C_i &  C_i^* D_i \\ D_i^* C_i & D_i^* D_i
\end{pmatrix}
= \sum_{i=1}^N \varphi\restriction_{\Mat_r} (C_i^*C_i)
\]
and $\varphi\restriction_{\Mat_r}$ is faithful, we obtain $C_i=0$ for all $i$. Moreover, it follows from Lemma \ref{lem:nilpotent} that $\tau\restriction_{p_\phi ' \Mat_d p_\phi '}$ is nilpotent and CP, and it follows from \cite[Cor.2.5]{BM} that the order of nilpotency is at most $\dim(p_\phi'\H) = d-r=:r'$. Therefore
\[
0 = \tau^{r'}\begin{pmatrix}
0 & 0 \\ 0 & \unit
\end{pmatrix}
= \sum_{i_1,\ldots, i_{r'}=1}^N \begin{pmatrix}
0 & 0 \\ 0 & D_{i_{r'}}^*\cdots D_{i_1}^* D_{i_1} \cdots D_{i_{r'}}
\end{pmatrix},
\]
so
\begin{equation}\label{eq:discrete-DD}
D_{i_1} \cdots D_{i_{r'}}=0,
\end{equation}
for all $i_1,\ldots,i_{r'}\in\{1,\ldots,N\}$. Moreover, unitality of $\tau'$ implies that
\begin{equation}\label{eq:discrete-AA}
\sum_{i=1}^N A_i^* A_i = \unit.
\end{equation}

We have $L_{i_1}\cdots L_{i_{k-1}}L_{i_k}=\begin{pmatrix}
  A_{i_1}\cdots A_{i_k}&M_{i_1,i_2,...,i_k}\\
  0&D_{i_1}D_{i_2}\cdots D_{i_k}
\end{pmatrix}$ for every $k\in\N$, where
\begin{align*}
M_{i_1,i_2,...,i_k} = &A_{i_1}\cdots A_{i_{k-1}}B_{i_k}+A_{i_1}\cdots A_{i_{k-2}}B_{i_{k-1}}D_{i_k}+ \ldots \\
&+ A_{i_1} B_{i_2} D_{i_3}\cdots D_{i_k} + B_{i_1} D_{i_2} \cdots D_{i_k}.
\end{align*}
Now it follows from \eqref{eq:discrete-DD} and \eqref{eq:discrete-AA} that
\[
\tau^{r'} \begin{pmatrix}
\unit & x \\ y & z
\end{pmatrix}
= \sum_{i_1,\ldots, i_{r'}=1}^N \begin{pmatrix}
A_{i_{r'}}^* \cdots A_{i_1}^* A_{i_1}\cdots A_{i_{r'}} &
A_{i_{r'}}^* \cdots A_{i_1}^*M_{i_1,i_2,\ldots,i_{r'}}\\
M_{i_1,i_2,\ldots,i_{r'}}^* A_{i_1}\cdots A_{i_{r'}} &
M_{i_1,i_2,\ldots,i_{r'}}^* M_{i_1,i_2,\ldots,i_{r'}}
\end{pmatrix}.
\]
Furthermore,
\begin{align*}
\sum_{i_0,i_1,\ldots, i_{r'}=1}^N M_{i_0, i_1,i_2,\ldots,i_{r'}}^* M_{i_0,i_1,i_2,\ldots,i_{r'}}
=& \sum_{i_0, i_1,\ldots, i_{r'}=1}^N M_{ i_1,i_2,\ldots,i_{r'}}^* A_{i_0}^* A_{i_0} M_{i_1,i_2,\ldots,i_{r'}} \\
=& \sum_{i_1,\ldots, i_{r'}=1}^N M_{i_1,i_2,\ldots,i_{r'}}^* M_{i_1,i_2,\ldots,i_{r'}}.
\end{align*}
Similarly
\begin{align*}
\sum_{i_0,i_1,\ldots, i_{r'}=1}^N M_{i_0, i_1,i_2,\ldots,i_{r'}}^* A_{i_0}A_{i_1}\cdots A_{i_{r'}}
=& \sum_{i_0, i_1,\ldots, i_{r'}=1}^N M_{ i_1,i_2,\ldots,i_{r'}}^* A_{i_0}^* A_{i_0} A_{i_1}\cdots A_{i_{r'}} \\
=& \sum_{i_1,\ldots, i_{r'}=1}^N M_{i_1,i_2,\ldots,i_{r'}}^* A_{i_1}\cdots A_{i_{r'}}.
\end{align*}
By induction we find that
\[
\tau^{r'+k} \begin{pmatrix}
\unit & x \\ y & z
\end{pmatrix}
= \tau^{r'} \begin{pmatrix}
\unit & x \\ y & z
\end{pmatrix}
= \begin{pmatrix}
\unit & 0 \\ 0 & \unit
\end{pmatrix},\quad \forall k\in\N,
\]
and together with \eqref{eq:discrete-n'} we see that $n'$ can be at most $r'$, so the order of $\tau$ on $\Mat_d$ can be at most $r^2-1+r'=r^2-1+d-r$.

It remains to show that all orders $n=2,\ldots, r^2-1+d-r$ can be attained. First of all, following the ideas in (Step 1) and given a root of order $n=1,\ldots,r^2-1$ on $\Mat_r$, there is $l=1,\ldots,r$ and $w\in\Mat_r$ such that
\[
\Big(\sum_{i_1,\ldots, i_{n-2}=1}^N A_{i_{n-2}}^* \cdots A_{i_1}^* w A_{i_1}\cdots A_{i_{n-2}}\Big)_{ll} \not=
\Big(\sum_{i_1,\ldots, i_{n-1}=1}^N A_{i_{n-1}}^* \cdots A_{i_1}^* w A_{i_1}\cdots A_{i_{n-1}}\Big)_{ll}.
\]
Then setting all $D_i=0$ and $B_i=e_{l,i}$ for $i=1,\ldots r'$, we can obtain roots of orders $n+1=2,\ldots,r^2$ on $\Mat_d$. In order to get order $n=r^2+n'$, we keep $B_i=e_{l,i}$ for $i=1,\ldots r'$ and choose for $D_1$ any contractive nilpotent matrix of order $n'+1$ and all other $D_i=0$. This way we achieve
\begin{align*}
\sum_{i_1,\ldots, i_{n'+1}=1}^N M_{i_1,i_2,\ldots,i_{n'+1}}^* M_{i_1,i_2,\ldots,i_{n'+1}}
=& \sum_{i_1,\ldots, i_{n'}=1}^N M_{i_1,i_2,\ldots,i_{n'}}^* M_{i_1,i_2,\ldots,i_{n'}}\\
\not=& \sum_{i_1,\ldots, i_{n'-1}=1}^N M_{i_1,i_2,\ldots,i_{n'-1}}^* M_{i_1,i_2,\ldots,i_{n'-1}},
\end{align*}
so in total we have a root $\tau$ of order $r^2+n'$, completing the proof of the theorem.
\end{proof}

We can adapt the construction in the preceding proof to obtain the corresponding statement in $B(\H)$ as follows:
\begin{theorem}\label{th:discrete-inf-state}
Suppose $\H$ is infinite-dimensional separable and $\varphi$ is a normal state on $B(\H)$. Then $\phi=\varphi(\cdot)\unit$ has a proper $n$-th discrete root on $B(\H)$, for every $n\in\N$.
\end{theorem}

\begin{proof}
Let $r=\dim (p_\phi\H)$ and $r'=\dim (p_\phi'\H)$. We distinguish two cases.

\emph{Case $r'=\infty$}. Here we choose $\alpha$ as a contractive nilpotent CP map of order $n$ on $B(p_\phi'\H)$. We define
\[
\tau \begin{pmatrix}
w & x \\ y & z
\end{pmatrix}
:= \begin{pmatrix}
\varphi\restriction_{B(p_\phi\H)}(w)\unit & 0 \\ 0 & \alpha(z) + \varphi\restriction_{B(p_\phi\H)}(w)(\unit-\alpha(\unit))
\end{pmatrix}
\]
Then $\tau$ is a proper $n$-th discrete root.

\emph{Case $r'<\infty$}. Then $r=\infty$ and we may assume as in the proof of Theorem \ref{th:discrete-fin-state} that the density matrix is in diagonal form with respect to a fixed orthonormal basis $(e_i)$ of $\H$ and with entries $\lambda_1\ge \lambda_2\ge \ldots$. Consider the projection $p_n$ onto $\spa\{e_1,\ldots,e_n\}$. Then
\[
\varphi_n := \frac{1}{\varphi(p_n)} \varphi\restriction_{B(p_n\H)}
\]
defines a faithful state on $B(p_n\H)$. We may then proceed as in (Step 1) of the proof of Theorem \ref{th:discrete-fin-state} to find a nilpotent map $\alpha_n: B(p_n\H)\ra B(p_n\H)$ of order $n$ such that $\alpha_n(p_n)=0=\varphi_n\circ\alpha_n$. We rescale $\alpha_n$ by $\varphi(p_n)$ and extend it trivially to $B(\H)\ominus B(p_n\H)$ and denote the resulting normal nilpotent map by $\alpha$. Then
\[
\tau := \phi + \alpha
\]
is a proper $n$-th discrete root of $\phi$.
\end{proof}

\subsection{Classical probability theory -- proper discrete roots of states on finite-dimensional commutative von Neumann algebras}

We would like to briefly specialize our general findings to the case of finite classical probability spaces because also here we get some interesting results. Note that a map $\tau:\mathbb{C}^d\to\mathbb{C}^d$ is UCP if and only if $\tau$ is a stochastic matrix and a map $\varphi:  \mathbb{C}^d\to \mathbb{C}$ is a state  if and only if there is a probability vector $p=(p_1,p_2,...,p_d)'\in \mathbb{C}^d$ such that $\varphi(x)=\langle p, x\rangle$, for all $x\in \mathbb{C}^d$.

In this subsection, we will use the following special notation. For $x\in \C^n, y\in \C^m$ we define $|x\rangle \langle y|:=xy^*\in \Mat_{n,m}.$ For any $x=(x_1,...,x_d)'\in \C^d$ and $m<d$, we write $x^{(m)}:=(x_1,...,x_m)' \in \C^m$. We write $\unit$ for the unit matrix but also for the unit vector $(1,\ldots,1)'\in\C^d$. Sometimes we will add subscripts or superscripts to $0$ and $\unit$ in order to indicate the space on which it is acting but we try to avoid this when it is obvious from the context.

As according to Lemma \ref{lem:discrete-rootnilpotent-finite-dim}, a state on $\mathbb{C}^d $ can have proper discrete roots only up to order $d-1$, the states on $\mathbb{C}$ and $\mathbb{C}^2$ will not have any proper discrete roots. The following example is a construction of proper $n$-th discrete roots of states  on $\mathbb{C}^d$, for all $2\le n\le d-1$ and $d>2$.

\begin{example}\label{ex:commutative} Let $d>2$ and $\varphi$ be a state on $\mathbb{C}^d$ given by a probability vector $p=(p_1,p_2,...,p_d)'$. Let  $\phi:=\varphi(\cdot)\unit.$ Then  $\phi$ is the stochastic matrix $\phi=|\unit\rangle \langle p|.$

First let us consider the  case when $\varphi $ is faithful, i.e., $p=(p_1,p_2,...,p_d)'$ with $p_i>0$, for all $i$. Let $2\le n\le d-1.$
Note that $\phi=|\unit\rangle \langle p|$ is diagonalizable and of rank one, so we can write
$\phi=S\begin{pmatrix}
1 &0 \\0 &0
\end{pmatrix}S^{-1}$, with a suitable invertible matrix $S$.
Consider a nilpotent matrix $\alpha_0\in \Mat_{d-1}$ of order $n$ and let
$\alpha=\eps S \begin{pmatrix}
0 &0\\ 0 & \alpha_0
\end{pmatrix}S^{-1}$.
If $\eps>0$ is small enough then all entries of $\phi+\alpha$ are non-negative because $\varphi$ was assumed to be faithful. By construction we have got $\phi\circ \alpha=0=\alpha\circ\phi$ and hence by Lemma \ref{lem:discrete-rootnilpotent}, $\tau=\phi+\alpha$ is a proper $n$-th discrete root of $\phi$.

Now let us assume that $\varphi$ is not faithful. Without loss of generality we can assume that $p=(p_1,p_2,...,p_r,0,...,0)', p_i>0$ for all $i=1,2,...,r<d.$ Let us consider two separate cases, namely $r\le 2$ and $2<r<d$, because our construction of $n$-th roots
works differently in these two cases.

\emph{Case $r\le 2$}. Given $r\le n\le d-1$, let
\[\tau =\begin{pmatrix}
|\unit^{(r)}\rangle \langle p^{(r)}| & 0 \\ |y(d-n)\rangle \langle e_1^{(r)}|&S_{n}
\end{pmatrix}\]
where
$y(d-n)=\begin{pmatrix}
\unit^{(d-n)} \\ 0^{(n-r)}
\end{pmatrix}\in \C^{d-r}$
and $S_{n}\in \Mat_{d-r}$ is the operator defined by $
S_{n}(e_i^{(d-r)})=0$ for $i=d-r,1,2,,...,d-n-1$ and
$S_{n}(e_i^{(d-r)})=e_{i+1}^{(d-r)}$ for $i=d-n,d-n+1,..., d-r-1$, and $e_i^{(d-r)}$ is the $i$-th canonical basis vector in $\C^{d-r}$.
Then $\tau$ is a proper $n$-th discrete root of $\phi$. (Note that
when $n=r$, we have $y(d-n)=\unit^{(d-r)}$ and $S_{n}=0$.)

\emph{Case $r>2$}.  Given $2\le n\le d-1$, decompose $n=n_1+n_2$,
with suitable $1\le n_1\le r-1$ and $1\le n_2\le d-r$. Let $
\tau_{[r,n_1]} $ be an $n_1$-th root of $|\unit^{(r)}\rangle \langle
p^{(r)}|$ as in the case of faithful $\varphi$ above. Then we define
\[\tau =\begin{pmatrix}
\tau_{[r,n_1]} & 0 \\ |y(d-n_2)\rangle \langle e_j^{(r)}|&S_{n_2}
\end{pmatrix}\] where $y(d-n_2)$ and $S_{n_2}$ are as in the previous case and $j$ is chosen as follows: if $n_1\ge 2$ then choose $j$ such that the $j$-th row of $\tau_{[r,n_1]}^{n_1-1}$ is different from $p^{(r)'}$, while for $n_1=1$ we choose $j=1$. Then $\tau$ is a proper $n$-th discrete root of $\phi$.

\end{example}

We summarize the result of the preceding example as follows:

\begin{theorem}
A state $\varphi$ on $\C^d$ has a proper $n$-th discrete root if and only if $2\le n\le d-1$. Or in more probabilistic terms: given a probability distribution $p$ on a probability space with $d$ elements, there is a stochastic map $S$ that leaves $p$ invariant and such that $S^n = |\unit\rangle \langle p|$ and $S^k \not= |\unit\rangle \langle p|$ for $k<n$ if and only if $2\le n\le d-1$.
\end{theorem}

Here $\varphi$ may be regarded as a stochastic matrix of rank $1$. For stochastic matrices of rank $>1$, we have no complete and simple characterization though some partial characterizations with necessary or sufficient conditions are known, e.g. in \cite{HL}. The case of rank $d$ is closely related to Elfving's embedding problem \cite{Elf,Dani}.

\section{Proper continuous roots}\label{sec:continuous}

We continue to use the notation from Section \ref{sec:discrete}.

\begin{definition}\label{def:continuous}
Given a von Neumann algebra $\A$ and a UNCP map $\phi:\A\ra \A$, a \emph{proper continuous root} of $\phi$ is a strongly-continuous one-parameter semigroup $(\tau_t)_{t\ge 0}$ of UNCP maps on $\A$ such that $\tau_1 = \phi$ and $\tau_t\not=\phi$, for all $0<t<1$.
\end{definition}

In this definition one might also consider seemingly more general semigroups with $\tau_{t_0}=\phi$ for some $t_0>0$. However, since we can always reduce the situation to the case $t_0=1$ by rescaling, we decided to keep things simple and consider only the case $t_0=1$. For more information on strongly continuous one-parameter semigroups in general, we refer the reader to \cite{Arv-book,Dav}.

\begin{proposition}\label{prop:root-bijective}
Let $\A$ be a finite-dimensional von Neumann algebra and $\phi:\A\ra\A$ a UNCP map. Then the following are equivalent:
\begin{itemize}
    \item[(i)] $\phi$ has a proper continuous root;
    \item[(ii)] $\phi$ is bijective and has a proper $n$-th discrete root, for every $n\in\N\setminus\{1\}$.
\end{itemize}
\end{proposition}

\begin{proof}
(i) $\Rightarrow$ (ii). If $(\tau_t)_{t\ge 0}$ is a proper continuous root, then it must be a uniformly continuous UNCP semigroup, hence of the form $\tau_t=\rme^{t\L}$ with some (bounded) conditionally completely positive generator $\L$, \cf \cite[Sec.4.5]{EK}, so $\rme^{-\L}$ is an inverse of $\phi$ (in the sense of linear maps on $\A$) and $\tau_{1/n}$ is a proper $n$-th discrete root of $\phi$, for every $n\in\N$.

(ii) $\Rightarrow$ (i). If $\phi$ has a proper $n$-th discrete root for every $n\in\N$ (this is called infinitely divisible in \cite{Den}) then according to \cite[Cor.4]{Den} there are a conditional expectation $E:\A\ra\A$ and a conditionally completely positive generator $\L$ such that $\phi=\rme^{\L} E$. Since $\phi$ and $\rme^{\L}$ are invertible, so is $E$ and hence $E$ must be the identity map because $\A$ is finite-dimensional. Thus we may choose $\tau_t=\rme^{t\L}$, for all $t\ge 0$, to obtain a proper continuous root of $\phi$.
\end{proof}

\begin{remark}
In the classical case, namely if $\A$ is commutative, $\phi$ is automatically bijective if it has a proper $n$-th discrete root for every $n\in\N$. This is one of the characterizations of Markovianity in the context of Elfving's embedding problem due to Kingman \cite[Prop.7]{Kin}. On the other hand, in the noncommutative case, bijectivity is not automatic. E.g. consider
\[
\phi:\Mat_3\ra\Mat_3, \quad \phi(x)=\begin{pmatrix}
x_{11} & 0 & 0\\
0 & x_{22} & x_{23}\\
0 & x_{32} & x_{33}
\end{pmatrix}.
\]
This has proper $n$-th roots for all $n$ but is clearly not bijective.
Similarly, we see that the UNCP map
\[
\phi:\Mat_2\ra\Mat_2, \quad \phi
\begin{pmatrix}
a & b\\
c & d\\
\end{pmatrix}
=
\begin{pmatrix}
d & c/2\\
b/2 & a\\
\end{pmatrix}
\]
from Example \ref{ex:discrete-UCP4} is bijective but has proper $n$-th roots only for odd $n\in\N\setminus\{1\}$, hence it has no proper continuous root.
\end{remark}

The following example provides a bijective UNCP map in finite dimensions where the conditions in the proposition are verified. In fact, it is a simple ``interpolation" of the construction in Example \ref{ex:discrete-UCP3}:

\begin{example}\label{ex:cont-UCP}
 Let $\phi:\Mat_2\ra\Mat_2$ be the UNCP map defined by
 \[
 \phi
\begin{pmatrix}
 a&b\\c&d
\end{pmatrix}
=\begin{pmatrix}
a&\frac{b}{2}\\\frac{c}{2}&d
\end{pmatrix}.
\]
For every $t\in[0,\infty)$, define
\[
\tau_{t}:\Mat_2\ra\Mat_2,\quad
\tau_{t}\begin{pmatrix}
 a&b\\c&d
\end{pmatrix}
=\begin{pmatrix}
a&\frac{b}{2^t}\\\frac{c}{2^t}&d
\end{pmatrix}.
\]
Then $(\tau_t)_{t\ge 0}$ is a proper continuous root of $\phi$, namely $\tau_1=\phi$ and the semigroup property and continuity are a straight-forward verification.
\end{example}

Embedding this example into a higher (possibly infinite) dimensional space, we can get continuous roots for certain UNCP maps in higher dimensions as well. A more complete criterion as to when such continuous roots exist seems out of reach. Notice that this might be even more difficult than Problem \ref{prob:discrete}.

Yet if $\phi$ arises from a state, we can say a little bit more:

\begin{proposition}\label{prop:cont-state}
Let $\A$ be a von Neumann algebra, $\varphi$ a state on $\A$ and $\phi=\varphi(\cdot)\unit$. If $(\tau_t)_{t\ge 0}$ is a proper continuous root of $\phi$ then
\begin{itemize}
\item[(i)] $\varphi\circ\tau_t = \varphi$, for every $t\ge 0$, i.e., $\varphi$ is $\tau$-invariant;
\item[(ii)] $\psi\circ\tau_t=\phi$, for every $t\ge 1$ and every UNCP map $\psi$, i.e., all UNCP maps composed with $\tau_t$ converge to $\phi$ in finite time; in particular, $\tau_t=\phi$, for all $t\ge 1$.
\end{itemize}
\end{proposition}

\begin{proof}
(i) Since $\phi=\tau_1$, for all $t\ge 1$, we get from the linearity and the semigroup properties of $\tau$:
\[
\varphi\circ\tau_t (x) \unit = \tau_1\circ \tau_t (x) = \tau_{t+1}(x) = \tau_t \circ \tau_1 (x)
= \tau_t( \varphi(x) \unit) = \varphi(x) \tau_t(\unit) = \varphi(x) \unit, \quad x\in \A.
\]

(ii) For all $t\ge 1$ and $x\in \A$, we have, using the unitality and the semigroup property of $\tau$:
\[
\psi \circ \tau_t (x) = \psi \circ\tau_{t-1}(\tau_1(x)) = \psi\circ\tau_{t-1} (\varphi(x) \unit) = \varphi(x) \psi\circ\tau_{t-1}(\unit) = \varphi(x) \psi(\unit)
= \phi(x).
\]
\end{proof}

The property that $(\tau_t)_{t\ge 0}$ stabilizes after time $t=1$ is very particular to states, \cf Example \ref{ex:cont-UCP} for a counter-example. In the special case where $\phi$ arises from a state and moreover $\A=B(\H)$, we can provide a partial classification of proper continuous roots:

\begin{theorem}\label{th:cont-state}
Let $\A=B(\H)$ with $\H$ infinite-dimensional, $\varphi$ a normal state on $\A$ and $\phi=\varphi(\cdot)\unit$.
\begin{itemize}
\item[(i)] If $\dim (p_\phi\H)=1$, i.e.~$\varphi$ is a pure state, then $\phi$ has a proper continuous root.
\item[(ii)] If $1<\dim (p_\phi\H) <\infty$, i.e.~$\varphi$ is a finite convex combination of (at least two) pure states, then $\phi$ has no proper continuous root.
\item[(iii)] If $\dim (p_\phi\H)=\infty$, i.e.~$\varphi$ is an infinite convex combination of pure states, and moreover $0<\dim (p_\phi'\H) <\infty$ then $\phi$ has no proper continuous root.
\end{itemize}
\end{theorem}

\begin{proof}
(i). This is taken from \cite[Ex.1.3]{Bhat12} and included here for the sake of completeness. Since $\varphi$ is pure, we can write $\varphi= \langle \xi, \cdot \xi\rangle$, where $\xi$ is a suitable vector in $\H$. We decompose $\H = \C\xi \oplus L^2[0,1]$, so $p_\phi$ is the projection onto the first, $p_\phi'$ the projection onto the second component. Let $(S_t)_{t\ge 0}$ be the standard nilpotent right-shift semigroup on $L^2[0,1]$  defined as follows: for $f\in L^2[0,1]$, $t\in [0,\infty)$ and $s\in[0,1]$, set
\begin{equation}\label{eq:cont-shift}
S_t(f)(s) =\left\{
\begin{array}{l@{\; :\;}l}
f(s-t) & s-t\in [0,1]\\
0 & \textrm{otherwise}.
\end{array}
\right.
\end{equation}

Then with respect to the decomposition $\H = \C\xi \oplus L^2[0,1]$, define
\[
\tau_t: B(\H) \ra B(\H),\quad  \begin{pmatrix}
x_{11} & x_{12} \\ x_{21} & x_{22}
\end{pmatrix}
\mapsto
\begin{pmatrix}
x_{11} &  x_{12} S_t^* \\
S_t x_{21} & S_t x_{22} S_t^* + x_{11} (\unit-S_tS_t^*)
\end{pmatrix}.
\]
This can be written as
\[
\tau_t \begin{pmatrix}
x_{11} & x_{12} \\ x_{21} & x_{22}
\end{pmatrix}
=
\tau_t(x) = (1\oplus S_t) x (1\oplus S_t)^* + \varphi(x) \big(\unit - (1\oplus S_t)(1\oplus S_t)^*\big),
\]
and it is straight-forward to verify that $(\tau_t)_{t\ge 0}$ is a strongly continuous semigroup, every $\tau_t$ is UNCP and $\tau_1(x)= \varphi(x) \unit=\phi(x)$. Thus $(\tau_t)_{t\ge 0}$ forms a proper continuous root of $\phi$.

(ii)   Suppose a proper continuous root $(\tau_t)_{t\ge 0}$ of $\phi$ existed. As in Lemma \ref{lem:nilpotent}(4) we see that $(p_\phi\tau_t(\cdot)p_\phi)_{t\ge 0}$ restricts to a continuous root of $p_\phi\phi(\cdot)p_\phi$ on $p_\phi \A p_\phi$.  However, we know from Proposition \ref{prop:root-bijective} that such a continuous root cannot exist because $p_\phi\phi(\cdot)p_\phi$ is not bijective on $p_\phi \A p_\phi$, so we reach a contradiction. Thus $\phi$ cannot have a proper continuous root.

(iii) Suppose for contradiction a proper continuous root $(\tau_t)_{t\ge 0}$ of $\phi$ existed. Since $\tau_t(p_\phi')\le p_\phi'$ according to Lemma \ref{lem:nilpotent}(i), we see that $(\tau_t\restriction_{p_\phi'\A p_\phi'})_{t\ge 0}$ forms an NCP semigroup, and according to Lemma \ref{lem:nilpotent}(ii), it is nilpotent with $\tau_1\restriction_{p_\phi'\A p_\phi'} =0$. If $0<\dim(p_\phi')<\infty$, a CP semigroup must be of the form $(\rme^{t\L})_{t\ge 0}$ with bounded conditionally CP map $\L$. Then $\rme^{-\L}$ is the inverse of $\rme^\L$ (as a linear map), so we get $0=\tau_1\restriction_{p_\phi'\A p_\phi'} = \rme^\L \not= 0$, which is a contradiction, so $\phi$ cannot have a proper continuous root.

\end{proof}

\begin{problem}\label{prob:inf-state}
In the setting of Theorem \ref{th:cont-state}, does $\phi$ have a proper continuous root in the following two missing cases
\begin{itemize}
\item[(iv)] $\dim p_\phi = \infty$ with $\dim p_\phi'=0$;
\item[(v)] $\dim p_\phi = \infty$ with $\dim p_\phi'=\infty$?
\end{itemize}
We wish to point out that the two cases are equivalent, so it suffices to study (iv).
\end{problem}

\begin{remark}
In \cite{Bhat12}, the roots in case (i) of Theorem \ref{th:cont-state} have been completely classified in terms of $E_0$-semigroups in standard form, cf.~\cite{Arv-book} and \cite[Def.2.12]{Pow}.
\end{remark}

\begin{remark}
A similar construction can be used in order to get a proper continuous root $(T_t)_{t\ge 0}$ of a pure state on an uncountable classical probability space $C[0,1]$, namely consider
\[
T_t:C[0,1]\ra C[0,1], \quad T_t f(s) = \left\{
\begin{array}{l@{\; :\;}l}
f(s-t) & s-t\ge 0\\
f(0) & \mbox{otherwise}.
\end{array}
\right.
\]
A pure state on $C[0,1]$ corresponds to an evaluation functional $\ev_x$, with some $x\in[0,1]$. Then $\ev_x\circ T_t$ equals a pure state at all times $t\in[0,1]$, in particular $\ev_x\circ T_1=\ev_0$. In contrast, in the noncommutative case of $\A=B(\H)$ as in Theorem \ref{th:cont-state}(i) suppose $\psi\not=\varphi$ is another pure state. Then $\psi\circ\tau_t$ equals the pure states $\psi$ at time $t=0$ and $\varphi$ at $t=1$ but in between it is a convex combination of two pure states depending on $t$. Moreover, for countable classical states space, we expect that no proper continuous root exists at all. This indicates a stark difference between the commutative and the noncommutative setting.
\end{remark}

\noindent\textbf{Acknowledgments.} BVRB thanks S.~Kirkland for a mathematical idea which helped us to construct Example \ref{ex:commutative}. We also thank T.~Cubitt and M.~Skeide for helpful comments on a former version of this manuscript. BVRB furthermore acknowledges financial support from  J.C.~Bose Fellowships. RH thanks the Indian Statistical Institute for the hospitality he received during research visits. NM thanks the Department of Atomic Energy, Government of India, for financial support and the IMSc Chennai for providing the necessary facilities. VU thanks the National Board for Higher Mathematics, India, for his PhD fellowship.

\end{document}